\documentclass[12pt]{extarticle}
\usepackage{amsmath}
\usepackage{amssymb}
\usepackage{stmaryrd}

\usepackage{ wasysym }
\usepackage{color}

   \newtheorem{theorem}{Theorem}[section]
   \newtheorem{lemma}[theorem]{Lemma}
   \newtheorem{proposition}[theorem]{Proposition}
   \newtheorem{corollary}[theorem]{Corollary}

\newtheorem{conjecture}{Conjecture}

\newcommand{\ben}{\begin{enumerate}}
\newcommand{\een}{\end{enumerate}}

\newcommand{\bt}{\begin{theorem}}
\newcommand{\et}{\end{theorem}}
\newcommand{\bl}{\begin{lemma}}
\newcommand{\el}{\end{lemma}}
\newcommand{\bc}{\begin{corollary}}
\newcommand{\ec}{\end{corollary}}
\newcommand{\bp}{\begin{proposition}}
\newcommand{\ep}{\end{proposition}}
\newcommand{\br}{\begin{remark}}
\newcommand{\er}{\end{remark}}

\newcommand{\bpf}{\begin{proof}}
\newcommand{\epf}{\end{proof}}

\newcommand{\be}{\begin{equation}} 
\newcommand{\ee}{\end{equation}}
\newcommand{\beq}{\begin{eqnarray}}
\newcommand{\eeq}{\end{eqnarray}}
\newcommand{\ba}{\begin{array}}
\newcommand{\ea}{\end{array}}
\newcommand{\bi}{\begin{itemize}}
\newcommand{\ei}{\end{itemize}}

\newcommand{\comm}[1]{}

\newcommand \qed {\hskip 6pt\vrule height6pt width5pt depth1pt \bigskip}

\newfont{\msbm}{msbm10 scaled\magstep1}
\newfont{\msbms}{msbm7 scaled\magstep1} 

   \newenvironment{proof}[1][Proof]{\begin{trivlist}
   \item[\hskip \labelsep {\bfseries #1}]}{\end{trivlist}}

   \newenvironment{remark}[1][Remark]{\begin{trivlist}
   \item[\hskip \labelsep {\bfseries #1}]}{\end{trivlist}}

   \comm{\newcommand{\qed}{\nobreak \ifvmode \relax \else
      \ifdim\lastskip<1.5em \hskip-\lastskip
      \hskip1.5em plus0em minus0.5em \fi \nobreak
      \vrule height0.75em width0.5em depth0.25em\fi}}

 \numberwithin{equation}{section}
 
\newcommand{\tr}{  \textrm{tr\ }  }

\begin{document}

\title{The Howland-Kato Commutator Problem, II}
\author{Richard Froese\\University of British Columbia, Vancouver\and Ira Herbst\\University of Virginia, Charlottesville}

\maketitle

\begin{abstract}
We continue the search, begun by Kato, for
all pairs of real, bounded, measurable functions $\{f,g\}$ that result
in a positive commutator $[if(P),g(Q)]$. We prove a number of partial
results including a connection with Loewner's celebrated
theorem on matrix monotone  functions.
\end{abstract}

\section{Introduction}
This is the second in a series of papers on the Howland-Kato commutator problem.  We refer the reader to the original paper of Kato \cite {TK} and the paper of Herbst and Kriete \cite{HK}. In a paper on an entirely different subject \cite{JSH}, Howland began the subject of this paper when he found a pair of functions,  $f = \tan^{-1}t/2$ and $g=\tanh t$ such that $i[f(P),g(Q)] \ge 0$.  Evidently this influenced Kato's attempt to find \emph{all} pairs of bounded functions $\{f, g\}$ which result in a non-negative commutator.\\
In \cite{TK}, Kato assumed that $f$ and $g$ were real and absolutely continuous with $L^1$ derivatives.  He proved a number of interesting facts with these assumptions. Let $K$ be the commutator $K=[if(P),g(Q)]$.  He showed that if $K=0$ then at least one of $f$ or $g$ is constant.  He showed that $K$ is trace class.  He solved the problem completely in the case when $K$ is rank one.  Specifically, he showed that in this case
$$
f(t) = c_1\tanh \hat\alpha(t-t_1) + d_1,
\quad g(t) = c_2\tanh \alpha(t-t_2) + d_2;
\quad
 $$
where $ \alpha \hat\alpha = \pi/2, \quad c_1c_2 > 0$, and
 $t_1$ and $t_2$ are arbitrary real numbers.  He did this by reducing the problem to a differential equation which he then solved.  Given Kato's solution in the rank one case, by linearity it follows that the functions
\begin{equation} \label{tanhrep}
f(t) = \int \tanh \hat\alpha(t-s)d\mu(s) + d_1, \quad g(t) =  \int \tanh \alpha(t-s)d\nu(s) + d_2
\end{equation}
result in a positive commutator whenever $\alpha \hat{\alpha} = \pi/2$ and $\mu$ and $\nu$ are finite positive measures.  Kato then made the important realization that a bounded real function had the representation which $f$ has in (\ref{tanhrep}) if and only if $f\in K_{\alpha}$, where $K_\alpha$ is the set of bounded, real functions with an analytic continuation to the strip $|\mathrm{Im}z| < \alpha$, satisfying $\mathrm{Im}f(z)\mathrm{Im}z \ge 0$.  If this representation looks unfamiliar in a complex variables setting, think of the Dirichlet problem for the imaginary part of an analytic function in a strip.  This results in a Poisson formula with a hyperbolic function integrated against the imaginary part on the boundary of the strip (see Widder \cite {W}).  In the limit where there is no continuation beyond the strip, the positive imaginary part becomes a finite measure.  Kato felt there was ``some reason to conjecture" \cite{TK} that all pairs of $f$ and $g$ resulting in a positive commutator would be as in (\ref{tanhrep}) (except that we could take $-f$ and $-g$ instead).   We will call this the Kato conjecture but we drop Kato's requirement that $f$ and $g$ be absolutely continuous with $L^1$ derivatives and replace it with the condition that $K \ne 0$.  The reason for this is that $K=0$ is special. As shown in  \cite {HK}, $f$ and $g$ are bounded and measurable with $K=0$ if and only if one of them is almost everywhere constant or they both have periodic versions with periods satisfying $\tau_f \tau_g = 2\pi$.  On the other hand:

\begin{theorem}\label{derivs}
If $f$ and $g$ are real, bounded, and measurable with $K \ge 0$ and $K\ne 0$, then they are both absolutely continuous with $L^1$ derivatives.  
\end{theorem}
Thus we can prove Kato's assumption if $K\ne 0$. The proof will be given in Section 2. \\
Since we will refer to it several times in this paper, we spell out the Kato conjecture:
\begin{conjecture}[Kato conjecture]
Given bounded, real, measurable functions, $f$ and $g$ with $K=i[f(P),g(Q)]\ge 0, K\ne 0,$ then  $\pm f$ and $\pm g$ have the representations given in (\ref{tanhrep}) with $\alpha \hat{\alpha} = \pi/2$.  Equivalently, $\pm f$ and $\pm g $ are in $K_{\alpha}$ and $K_{\hat{\alpha}}$ respectively.
\end{conjecture}
The $\pm$ are correlated. Whenever we have a choice we will choose the plus signs.
There are two results we want to discuss which provide evidence for the truth of Kato's conjecture.  The first involves matrix monotone functions whose structure was discovered by Loewner \cite{L}.  These are functions $F$ which have the property that if $I$ is an open interval on which $F$ is defined then for any two self adjoint operators $A$ and $B$ with spectrum in $I$ and which satisfy $A\ge B$, we have $F(A) \ge F(B)$. Loewner showed that $F$ has this monotonicity property for operators on an infinite dimensional Hilbert space if and only if $F$ has an analytic continuation to $I\cup \{\mathrm{Im} z \ne 0\}$ with $\mathrm{Im}F(z)\mathrm{Im}z \ge 0$.  In particular $F$ is a Herglotz function.  We have 

\begin{theorem} \label{monotone}
Suppose $f$ and $g$ are real, bounded, measurable functions with ranges $I$ and $J$ respectively. Suppose $ K = i[f(P),g(Q)] \ge 0, K \ne 0$.  If $F$ and $G$ are matrix monotone functions defined on open intervals containing the closure of $I$ and $J$ respectively, then $i[(F\circ f)(P), (G\circ g)(Q)] \ge 0$.
\end{theorem}
Theorem \ref{monotone} will be proved in Section 3.\\
The significance of this result is that if $g\in K_{\alpha}$ then by the Herglotz property of matrix monotone functions, $G\circ{g} \in K_{\alpha}$ and similarly if we replace $g$ and $G$ by $f$ and $F$. Thus these matrix monotone functions preserve the pairs $\{f,g\}$ which give a non-zero non-negative commutator as well as preserving the conjectured $K_{\alpha}$ and $K_{\hat{\alpha}}$.  \\

The second piece of evidence is a result from \cite{HK} which we now discuss.  First let us define the concept of maximality:  If $f$ is in $K_a$,  $a$ is \emph{maximal} for $f$ if $f$ is not in any $K_c$ with $c >a$.  A sufficient (but by no means necessary) condition guaranteeing that $a$ is maximal for $f$ is that $f\in K_a$ and for all $b\in (0,\hat a), \int f'(t) e^{2bt} dt < \infty$.  To see this note that the derivative of $\tanh$ is $\cosh^{-2}$ so that if $f(t) = \int \tanh \hat a(t-s) d\mu(s) + d$ then 

$$\int f'(t) e^{2bt}dt = \hat a \int \frac{e^{2bt}}{\cosh^2 \hat a t} dt \int e^{2\hat a s}d\mu(s).$$ 
Thus if $f$ were in $K_c$ with $c > a$ (and thus $\hat c < \hat a$) the latter integral with $\hat c $ replacing $\hat a$ would diverge for $b$ close enough to $\hat a$.  We can now state the next theorem which is essentially given in \cite{HK}:  

\begin{theorem}\label{assumehalf}
Suppose $f\in K_a$ and  $\int f'(t) e^{2bt} dt < \infty$ for all $b\in (0, \hat a)$ (so that $a$ is maximal for $f$).  If $ K = i[f(P),g(Q)] \ge 0, K\ne 0$ then $g\in K_{\hat a}$.   
\end{theorem}
Our next result may seem a bit surprising in view of Theorem \ref{assumehalf}, but considering that we believe the Kato conjecture is true, it is just another step showing that we might be right:

\begin{theorem}\label{expdecay}
In Theorem \ref{assumehalf} the result is true if we drop the assumption that $f\in K_a$.  In other words, if $K = i[f(P),g(Q)] \ge 0, K\ne 0$ and $\int f'(t) e^{2bt} dt < \infty$ for all $b\in (0, \hat a)$, then $g\in K_{\hat a}$.  As a corollary, due to the symmetry between $f$ and $g$, if in addition  $\int g'(t) e^{2bt} dt < \infty$ for all $b\in (0, a)$, then $f \in K_a$ and thus under these exponential decay assumptions, the Kato conjecture is true. 
\end{theorem}
We prove Theorem \ref{expdecay} in Section 4. \\
We remark that the theorem is true if we replace exponential decay in the positive direction with exponential decay in the negative direction, or in other words if we replace $\int f'(t) e^{2bt} dt < \infty$ for all $b\in (0, \hat a)$ with $\int f'(t) e^{-2bt} dt < \infty$ for all $b\in (0, \hat a)$. 

In Theorem \ref{expdecay} we mentioned the symmetry between $f$ and $g$.  This will be used several times in this work and it is just a consequence of the unitary equivalence of $Q$ and $P$.  If $\mathcal{F}$ is the Fourier transform, then $Q = \mathcal F P \mathcal F^{-1}$ and  $-P = \mathcal F Q \mathcal F^{-1}$ .  If we write 

$$\tilde K = \mathcal F K \mathcal F^{-1},$$ then 

\begin{equation}
\tilde K = i[-g(-P), f(Q)].  
\end{equation}

This symmetry is evident in the expression for the $x$-space and $p$-space expressions for the integral kernels of the operators $K$ and $\tilde K$:

\begin{equation} \label{kernels}
K(x,y) = (1/\sqrt{2\pi})\frac{g(x) - g(y)}{x-y}\widehat{f'}(y-x), \quad \tilde K(\xi,\eta) =  (1/\sqrt{2\pi})\frac{f(\xi) - f(\eta)}{\xi - \eta} \widehat{g'}(\xi-\eta)
\end{equation}



In Section 5 we turn our attention to the case where $K$ has finite rank.  We prove a number of results there, chief among them is:

\begin{theorem}\label{finiterank}
Suppose $i[f(P),g(Q)] = \sum_{k=1}^N (\phi_k,\cdot ) \phi_k $ where the $\phi_k$ are linearly independent $L^2(\mathbb{R})$ functions.
Then there exist positive $a$ and $b$ such that $f\in K_a$ and $g\in K_b$.  If $f$ is analytic in the strip $\{|\mathrm{Im}z| < c\}$ and bounded there, then $f\in K_c$.  If $f\in K_{\alpha}$ and $g\in K_{\beta}$
then $\alpha\beta \le \pi/2$.
\end{theorem}
There are a couple of things to note about this theorem.  The first is that the finite rank condition does not allow $f$ and $g$ to be in $K_\alpha$ and $K_\beta$ respectively if $\alpha \beta > \pi/2$.  If we remove the finite rank condition this is not true (as the original Howland example shows).  Secondly, although we have not proved Kato's conjecture in this case, the only thing missing is the analyticity and boundedness in large enough strips.  The positivity of the imaginary part in the upper half strip comes for free once this analyticity is known.  

The next theorem shows that given $K=i[f(P),g(Q)] \ne 0, K\ge 0$ and that $f$ and $g$ are bounded real and measurable, we can also assume in proving the Kato conjecture that $f\in K_a$ and $g\in K_b$ for some positive $a$ and $b$ and that both are entire.

\begin{theorem} \label{$K_a$ assumption}
Suppose we know that the Kato conjecture is true for entire functions $f$ and $g$ with $f\in K_a$ and $g\in K_b$ for some positive $a$ and $b$.  Then the Kato conjecture is true if we just assume that $f$ and $g$ are real, bounded, and measurable.
\end{theorem}
We prove Theorem \ref{$K_a$ assumption} in Section 6.
\section{Proof of theorem \ref{derivs}}  

\begin{theorem}
If $f$ and $g$ are real, bounded, and measurable with $K \ge 0$ and $K\ne 0$, then they are both absolutely continuous with $L^1$ derivatives.  
\end{theorem}
\begin{proof}
Two results important for the proof of this theorem were already proved in \cite{HK}, namely that $f$ and $g$ are  monotone and continuous (without loss of generality we assume they are both monotone increasing), and in addition that $K$ is trace class.  We refer the reader to the proofs in that paper.  The most difficult one is the monotonicity. Since we do not know the absolute continuity of $f$ we can substitute the equally valid formula

$$K(x,y) = (1/\sqrt{2\pi})\frac{g(x) - g(y)}{x-y}\widehat{df}(y-x)$$
for the expression in (\ref{kernels}).  Here $df$ is the positive measure associated with the monotone function $f$.  Let $\{\phi_k\}$ be an orthonormal set of eigenfunctions of $K$ spanning the space where $K>0$.  Suppose $K\phi_k = \lambda_k\phi_k$.  Then
$$K = \sum_k \lambda_k (\phi_k,\cdot) \phi_k $$
We use results and methods from Brislawn \cite{B1}  where he is interested in obtaining a substitute for the formula $\tr K =\int K(x,x) dx$ when we do not know that $K(x,y)$ is continuous.  From \cite {B1} we know that (with $C_r = [-r,r]$ and $|C_r| = 2r$)

$$ \lim_ {r \to 0} |C_r|^{-2} \int _{C_r \times C_r} K(x+t,x+s) dtds = \sum \lambda_i |\phi_i(x)|^2 \ a.e. $$ 
This implies that 
\begin{lemma}\label{k(x)}
\begin{equation}
\lim_{r \to 0} \int_{|t|<r,|s|<r}(g(x+t) - g(x+s))(t-s)^{-1} dt ds/(2r)^2 = :k(x)\ a.e.
\end{equation}
\end{lemma}
As we will show, $k$ will turn out to be the derivative of  $g$.

Let $[f] = f(\infty) - f(-\infty)$ and $$A_r \phi(x) = |C_r|^{-1} \int_{C_r}  \phi(x+t)dt,$$ $$A_r^{(2)}\psi(x,y) = |C_r|^{-2}\int_{C_r\times C_r}\psi(x+t,y+s) dtds.$$
For $x$ close to $y$ we can write $(2\pi)^{-1} [f] (g(x) - g(y))/(x-y) = K_1(x,y) := K(x,y) \widehat {df}(0)/\widehat{df}(y-x)$.  Thus 
$$\int(g(x+t) - g(x+s))(t-s)^{-1} dt ds/(2r)^2= (2\pi/[f]) [(A_r^{(2)}K)(x,x) - A_r^{(2)}(K-K_1)(x,x)$$
Since $(A_r^{(2)}K)(x,x)$ converges a.e., we need only show that $A_r^{(2)}(K-K_1)(x,x) \rightarrow 0, a.e.$ 
Given $\epsilon > 0 $, for small enough $r$, $|1- \widehat{df}(0)/\widehat{df}(t-s)| < \epsilon$ if both $|t|<r, |s|<r$.  Thus 
$$|A_r^{(2)}(K-K_1)(x,x)| \le \epsilon \sum \lambda_i (A_r|\phi_i|(x))^2 \le   \epsilon \sum \lambda_i (M\phi_i(x))^2.$$
where $M \psi$ is the maximal function. Since $||M\psi||_2  \le C||\psi||_2 , \int \sum \lambda_i (M\phi_i(x))^2 dx \le  C^2\sum \lambda_i < \infty$,  the sum above converges a.e.  This proves the result.
\end{proof}

We calculate further to elucidate the meaning of $k$.

\begin{proposition}
    $$k(x) = \lim_{r \to 0} \int_0^1 h(w) \frac{g(x+rw) - g(x-rw)}{2rw} dw$$
where $h(w)$ is an increasing function with $\int_0^1 h(w) dw =1$.  In fact 
$$h(w) = w \log(\frac{1+w }{1-w})$$
\end{proposition}

\begin{proof}
Define
$$I_r(x) = \int_{C_r\times C_r} \frac{g(x+t) - g(x+s)}{t-s} dt ds(2r)^{-2}.$$
We rewrite this as 
$$I_r(x) = (2r^2)^{-1}\int \int_{|t|<r}\int_{s<t, |s| <r} (t-s)^{-1} 1_{[x+s,x+t]}(u) ds dt dg(u).$$
Doing the $s$ integral and then the $t$ integral we have 
$$I_r(x) = (2r^2)^{-1}\int \int_{|t| < r} 1_{[x-r,x+t]}(u) \log \frac{t+r}{t-(u-x)} dt dg(u)= $$
$$ (2r)^{-1} \int 1_{[x-r,x+r]} (\log 4 - j(w)) dg(u) ,$$ $$ j(w) = (1+w)\log(1+w) + (1-w)\log(1-w), w = (u-x)/r.$$
Integrating by parts we find
$$I_r(x) = (2r)^{-1} \int ((d/du)j(w) )(g(u) - g(x) du = (2r)^{-1} \int_{-1}^1 j'(w) (g(x+rw) - g(x))dw = $$
$$\int_0^1 (w \log \frac{1+w}{1-w}) (\frac {g(x+rw) - g(x-rw)}{2rw}) dw.$$

We have $h(w) =w\log \frac{1+w}{1-w} = 2\sum_{m=1}^\infty w^{2m}/(2m-1)$ so that for $w>0$, $h(w)$ is positive and $h'(w) >0$.  $h(w) \uparrow \infty$ as $w\uparrow 1$.  Integration by parts gives $\int_0^1 h(w)dw = 1$.  
\end{proof}
\begin{lemma} \label{forLDCT}
$$0\le I_r(x) \le G(x)$$
for some integrable function $G$.
\end{lemma}
\begin{proof}
Choose $r_0$ so that if $r < r_0$,  $|\widehat{df}(0)/\widehat{df}(x)| < 2$ when $|x| < 2r_0$. As in the proof of Lemma \ref{k(x)}, we have 
$$0\le([f]/2\pi) I_r(x) = (A_r^{(2)}K_1)(x,x) \le 2(A_r^{(2)}K)(x,x) \le2\sum_i \lambda_i |A_r \phi_i(x)|^2 \le 2\sum_i \lambda_i |M \phi_i(x)|^2 = G(x).$$
Since $||M\phi_i||_2 \le C||\phi_i||_2 = C$, $G$ is integrable.
\end{proof}
\begin{proposition}
The $L^1$ function $k$ is the distributional derivative of $g$.  In fact if $\psi$ is absolutely continuous with $\psi' \in L^1$,
$$ (\psi g)(-\infty) - (\psi g)(+\infty) + \int \psi'(x) g(x) dx = - \int \psi(x) k(x) dx.$$
\end{proposition}
Note that both $\psi$ and $g$ have limits at $\pm \infty$.
\begin{proof}
Using Lemma \ref{forLDCT}, if  $\psi$ is absolutely continuous with $\psi' \in L^1$ , $\lim_{r\to 0} \int I_r(x) \psi(x) dx = \int k(x) \psi(x) dx$.  To handle the left side of the equation, first suppose $\lim_{x \to \pm \infty} \psi(x) = 0$.
We then have $$\lim_{r\to 0} \int I_r(x) \psi(x) dx = \lim_{r\to 0} \int\int h(w) (g(x+ rw) - g(x-rw))(2rw)^{-1} \psi(x) dx dw=$$
$$- \lim_{r\to 0}\int h(w)\int ( \psi(x+rw) - \psi(x-rw))(2rw)^{-1}g(x) dx dw =$$
$$- \lim_{r\to 0}\int h(w)\int (2rw)^{-1} \int_{-rw}^{rw} \psi'(x+y)dy g(x)dx dw= $$
$$- \lim_{r\to 0}\int h(w) \int (2rw)^{-1} \int_{-rw}^{rw} g(x-y)dy \psi'(x)dx dw= - \int g(x)\psi'(x)dx.$$
We have used $\lim_{x \to \pm \infty} \psi(x) = 0$ to obtain the second equality.  Now drop the latter assumption and let $\lambda_\epsilon \in C_0^\infty$ with $\lambda_\epsilon(x) =\int_{-\infty}^x (\phi_\epsilon(t-L ) - \phi_\epsilon(t+L ))dt$ where $\phi_\epsilon(t) = \epsilon^{-1} \phi(t/\epsilon)$ and  $\phi \in C_0^\infty(\mathbb{R})$ satisfies $0\le \phi(t)$ with $\int \phi(t) dt = 1$.  We now have 
$$\int(\lambda_\epsilon \psi)'(x)g(x)dx = -\int \lambda_\epsilon(x)\psi(x)k(x) dx.$$
Clearly $\int \lambda_\epsilon \psi' g dx \to \int_{-L}^L \psi'g dx$ and $\int \lambda_\epsilon \psi k dx \to \int_{-L}^L \psi k dx$ as $\epsilon \to 0$.  It is also easy to see that $\int \lambda_\epsilon' \psi g dx \to (\psi g)(-L) - (\psi g)(L)$.  The result follows after taking $L\to \infty$.
\end{proof}

\begin{theorem}\label{gac}
    $g$ is absolutely continuous and $g' = k$ , a.e.
\end{theorem}

\begin{proof}
Fix $t$ and let $\psi_\epsilon(x) = 1-\int_{-\infty}^x \phi_\epsilon(y-t) dy$.   
$$(\psi_\epsilon g)(-\infty) -  (\psi_\epsilon g)(\infty)+ \int \psi_\epsilon'(x) g(x) dx = - \int \psi_\epsilon(x) k(x) dx.$$
Then since $\psi_\epsilon(x) \to 1_{(-\infty,t]}(x)$ except at $x=t$, $0 \le \psi_\epsilon \le 1$, and $g$ is continuous, bounded and increasing, taking $\epsilon \to 0$ gives
$$g(-\infty) -g(t) =   -\int_{-\infty}^t k(x) dx$$
\end{proof}
This gives the result.\\
The first part of the following Corollary is known from \cite {HK}.   
\begin{corollary} \label{positivederivs}
With our standard assumptions including $K\ne 0$, $f$ and $g$ are strictly monotone.  More precisely $f'$ and $g'$ are positive a.e.
\end{corollary}

\begin{proof}
   Suppose $g(a) = g(b)$ where $b > a$.  Then since $g$ is monotone, $g(x) - g(y) = 0$ if $x,y \in [a,b]$.  Thus $K\psi(x) = 0, a.e. \  x\in  [a,b]$ if $\mathrm{supp}\ \psi \subset [a,b]$.  Thus $(\psi,K\psi) =0$.  Since $K$ is positive this implies $K\psi = 0$ and since this is true for all $\psi$ with $\mathrm{supp} \ \psi \subset [a,b]$ , $K(x,y) = 0$ for $y\in [a,b]$ and $|x-y| >0$, in particular for $y\in [a,b]$ and $b< x < b+\delta_0$ or  $a-\delta_0 < x < a$ where $\delta_0$ is chosen so that $|\hat f'(u)| > 0$ if $|u| \le \delta_0$.  It follows that $g$ is constant in $[a,b] + (-\delta_0,\delta_0)$.  The strict monotonicity then follows by induction.

From \cite{HK} it follows that $g^{-1}$ is absolutely continuous.  Thus if $x$ is a point where $g$ is differentiable and where $g^{-1}$ is differentiable at $g(x)$ then $(g^{-1})' (g(x)) g'(x) = 1$.  So $g'(x) > 0$.  If $A = \{x: g'(x) \text{exists}\}$ and $B=\{y:(g^{-1})'(y) \text{exists}\}$ then $g'(x) > 0$ for $x \in A\cap g^{-1}(B)$.  We have $g: \mathbb{R} \rightarrow I$, $I = (g(-\infty), g(\infty))$ and 
$g^{-1}: I \rightarrow \mathbb{R}$.  Then $|(A\cap g^{-1}(B))^c| = |A^c \cup (g^{-1}(B))^c| \le |A^c| + |(g^{-1}(B))^c| $.  Of course $|A^c| = 0$ and $(g^{-1}(B))^c = g^{-1}(I \setminus B)$. Since $|I\setminus B| = 0$ and $g^{-1}$ is absolutely continuous, this has measure zero.  This completes the proof.
\end{proof}

\section{Commutator positivity and matrix monotonicity}

In this section we will prove Theorem \ref{monotone}.

We have 
$$d/dt(e^{itf(P)} g(Q) e^{-itf(P)})  = e^{itf(P)} i[f(P),g(Q)]e^{-itf(P)} \ge 0.$$
Thus $$e^{itf(P)} g(Q) e^{-itf(P)} \ge g(Q)$$
for $t\ge 0$.  If $G$ is a matrix monotone function defined on an open set containing the closure of the range of $g$ then for $t\ge 0$

$$ G(e^{itf(P)} g(Q) e^{-itf(P)})  \ge G(g(Q))$$
or 
$$e^{itf(P)}G(g(Q))e^{-itf(P)} \ge G(g(Q)) $$
It follows that the derivative 
$$d/dt e^{itf(P)}G(g(Q))e^{-itf(P)} |_{t=0} \ge 0$$
or $$i[f(P),(G\circ g)(Q)] \ge 0.$$
Now with $G\circ g$  replacing $g$, if $F$ is a matrix monotone function defined on an open interval  containing the range of $f$ the same proof shows that 
$$i[(F\circ f)(P),(G\circ g)(Q)] \ge 0.$$
This proves the result.


\section{Proof of Theorem \ref{expdecay}}
Let us restate Theorem \ref{expdecay}.  For convenience we state the theorem with a change of symbols $\hat{a} \rightarrow a$.  Note that we have already shown that if $K\ge 0$ and $K\ne 0$, that $f$ and $g$ are absolutely continuous with derivatives positive a.e. (see Theorem \ref{derivs} and Corollary \ref{positivederivs}).  
\begin{theorem}
Suppose $K=i[f(P),g(Q)] \ge 0, K\ne 0$, $f$ and $g$ real, bounded, and measurable.  Suppose in addition $\int f'(\xi) e^{2b\xi} d\xi < \infty$ for $b\in (0,a) $.  Then $g\in K_a$.
\end{theorem}
We first prove an important formula:
\begin{proposition} \label{traceformula}
\begin{equation}
\widehat{f'}(y-x) =  (\sqrt{2\pi}/[g])\tr(e^{iPx}Ke^{-iPy}).  \label{widehat}
\end{equation}
\end{proposition}
Here $[g] = g(\infty) - g(-\infty)$.
\begin{proof}

First assume that $f$ is smooth with bounded derivatives of all orders.  Then 
\begin{equation}\label{tildekernel}
    \frac{1}{\sqrt{2\pi}} \frac{f(\xi) - f(\eta)}{\xi-\eta} \widehat{g'}(\xi-\eta)  = \tilde {K} (\xi,\eta)
\end{equation}
where $\tilde {K}$ represents the operator $K$ in the sense that $ \widehat{K \psi}(\xi)  = \int \tilde{K}(\xi,\eta) \hat {\psi}(\eta) d\eta.$  In other words, $\tilde K(\xi,\eta)$ is the integral kernel of the operator $\mathcal{F}K\mathcal {F}^{-1}$.  It is clear from this formula that $\tilde{K}$ is uniformly continuous.    We have 

\begin{equation}\label{diagg}
([g]/2\pi)f'(\xi) = \tilde{K}(\xi,\xi). 
\end{equation}
In particular we have $[f][g]/2\pi = \int \tilde{K}(\xi,\xi)d\xi = \tr K$. \\ 
We are using the fact that for a trace class operator $A$ with continuous integral kernel $G(x,y)$, we have $\tr A= \int G(x,x) dx $.  This is a special case of a result in \cite{D}.  For further generalizations see \cite{B1,B2}. \\
Multiplying both sides of (\ref{diagg}) by $e^{i\xi(x-y)} $ and integrating we get

$$([g]/2\pi) \int f'(\xi)e^{-i\xi(y-x)} d\xi = \int e^{i\xi x}\tilde{K}(\xi,\xi) e^{-i\xi y}d\xi $$ or

$$\hat{f'}(y-x) = (\sqrt{2\pi}/[g])\tr (e^{iP x}Ke^{-iPy}).$$
To get rid of the assumption of smoothness of $f$, we convolve $f$ with a non-negative function $\psi$ in the Schwartz space $\mathcal{S}$.
Then the equation takes the form $$\widehat{f'}(v-u)\sqrt{2\pi}\hat{\psi}(u-v) =  (\sqrt{2\pi}/[g])\int \tr (e^{iPu}e^{iQa}Ke^{-iQa}e^{-iPv}) \psi(a) da.$$ \\
Continuing we have 

 $$ \widehat{f'}(v-u)\sqrt{2\pi}\hat{\psi}(u-v) =  (\sqrt{2\pi}/[g])\int e^{ia(u-v)}\tr (e^{iPu}Ke^{-iPv}) \psi(a) da =  $$  $$(\sqrt{2\pi}/[g])\tr (e^{iPu}Ke^{-iPv}) \sqrt{2\pi}\hat{\psi}(v-u). $$
 Taking $\psi$ to be a Gaussian we have that $\hat{\psi} > 0$ which gives the result.

  \end{proof}

\begin{lemma} \label{propp}
Suppose $f$ and $g$ are as in Theorem \ref{expdecay}. Then 

$$\tr(K^{1/2} e^{bP} K^{1/2}) < \infty$$ for $0\le b <2a$ and  $$\widehat{f'}(z) =  (\sqrt{2\pi}/[g])\tr(K^{1/2} e^{-iPz}K^{1/2})$$ 
for $0 \le $ Im $z < 2a$. 

\end{lemma}

\begin{proof} (of Lemma \ref{propp})\\
We have $$\widehat{f'}(x) = \int e^{-i\xi x} f'(\xi)d\xi/\sqrt{2\pi} = $$ 
$$ (\sqrt{2\pi}/[g])\tr(e^{-iPx} K) = (\sqrt{2\pi}/[g])\tr(e^{-iPx} K^{1/2} K^{1/2}) = $$ 
$$ (\sqrt{2\pi}/[g])\tr(K^{1/2}e^{-iPx} K^{1/2}) . $$
Here we have used that $\tr(AB) = \tr(BA)$ when both $A$ and $B$ are Hilbert-Schmidt.  Multiplying by a function in $\mathcal{S}$ and integrating over $x$ gives  

$$\int \varphi (\xi) f'(\xi)d\xi = (2\pi/[g])\tr(K^{1/2}\varphi(P) K^{1/2}) . $$  

We take $\varphi(\xi) = e^{-\epsilon \xi^2} e^{b\xi}$ with $\epsilon >0$ and $0 < b < 2a$ and let $\epsilon \downarrow 0$.  By the monotone convergence theorem, since $\int e^{b\xi} f'(\xi)d\xi < \infty$ we have $\tr(K^{1/2} e^{bP} K^{1/2}) < \infty$.  It follows that $e^{bP/2} K^{1/2}$ is a Hilbert-Schmidt operator as is its adjoint which we write as $K^{1/2}e^{bP/2}$.  We can think of the latter operator as the closure of  $K^{1/2}e^{bP/2}$ with its natural domain. (We remark that since $1_{(-\infty,0]}(P) e^{bP}K^{1/2}$ is Hilbert-Schmidt, so is $1_{[0,\infty)}(P) e^{bP}K^{1/2}$.)   By analytic continuation 

$$\widehat{f'}(z) =  (\sqrt{2\pi}/[g])\tr(K^{1/2}e^{-iPz} K^{1/2}) $$
 for $0\le$ Im $z < 2a$.  
 
 \end{proof}
 
 \begin{proof} (of Theorem \ref{expdecay})\\
 Again using $\tr(AB) = \tr(BA)$ for two Hilbert-Schmidt operators, it follows that 
 $$\widehat{f'}(w-z) = (\sqrt{2\pi}/[g])\tr(K^{1/2} e^{iPz} e^{-iPw}K^{1/2}) = (\sqrt{2\pi}/[g])\tr(e^{-iPw}Ke^{iPz}) = (\sqrt{2\pi}/[g])\tr(e^{iPz}Ke^{-iPw})$$ 
 whenever $0\le$- Im $z < a$, $0 \le $ Im $w < a$.
 For these values of $w$ and $z$,  $\mathcal{F}e^{iPz}Ke^{-iPw}\mathcal{F}^{-1}$ is a trace class operator, thus Hilbert-Schmidt.  The integral kernel of this operator is 
 
 $$\mathcal{F}e^{iPz}Ke^{-iPw}\mathcal{F}^{-1}(\xi,\eta) = e^{i\xi z} \tilde{K}(\xi,\eta) e^{-i\eta w}$$ which is thus square integrable.  The $x$ space kernel of the latter operator is thus 
 
 $$\int e^{i\xi z} \tilde{K}(\xi,\eta) e^{-i\eta w} e^{i(x\xi -y\eta)} d\xi d\eta/(2\pi).$$  
 Before going further we regularize $g$.  Let $\psi(t) = (1/\sqrt{2\pi}) e^{-t^2/2}$, and $\psi_\epsilon(t) = \epsilon^{-1}\psi(t/\epsilon)$.  Define $g_\epsilon = \psi_\epsilon * g$ and note that  $g_\epsilon $ is an entire function.  This replaces the operator $K$ with $K_\epsilon =\int e^{-iPa}Ke^{iPa}\psi_\epsilon(a) da$ which has kernel $K_\epsilon(x,y) = \int K(x-a,y-a) \psi_\epsilon(a) da$.

From the definition of $g_\epsilon(w)$ we have 

$$K_\epsilon(z,w) = \frac{1}{\sqrt{2\pi}} \frac{g_\epsilon(z) - g_\epsilon(w)}{z-w} \widehat{f'}(w-z).$$  
We set $w=x+iy$, $y>0$, and $z = x-iy$ to obtain

$$K_\epsilon(x-iy,x+iy)= \frac{1}{\sqrt{2\pi}}(\text{Im}{g_\epsilon(x+iy)}/y) \widehat{f'}(2iy). $$  But 

$$K_\epsilon(x-iy,x+iy)= \int e^{ix\xi} e^{y\xi}K_\epsilon(\xi,\eta)e^{y\eta} e^{-iy\eta}d\xi d\eta/(2\pi) $$
$$= \lim_{\delta \downarrow 0} \int e^{ix\xi} e^{y\xi}e^{-\delta \xi^2}K_\epsilon(\xi,\eta)e^{-\delta \eta^2}e^{y\eta} e^{-iy\eta}d\xi d\eta $$
which is non-negative because of the non-negativity of the operator $K_\epsilon$. Thus Im $g_\epsilon(z) \ge 0 $ if $a>$ Im $z\ge 0$.  Since $g_\epsilon(\bar{z}) = \overline{g_\epsilon(z)}$, Im $g_\epsilon(z) \cdot \mathrm{Im} z \ge 0$ for all $z$ with $| \mathrm{Im} z| <a$.
As Kato has shown in \cite{TK}, this implies the existence of a finite positive Borel measure $\mu_\epsilon$ and a real number $d_\epsilon$ so that 

$$g_\epsilon(x) = \int \tanh \hat{a}(x-t) d\mu_\epsilon(t) + d_\epsilon$$
where $\hat{a} = \pi/2a$.  We now use a compactness argument from \cite{HK}.  Taking the limits $x\to \pm \infty$ we find 

$$g(\pm \infty) = \pm \mu_\epsilon(\mathbb{R}) + d_\epsilon$$
 so that 

$$g(\infty) - g(-\infty) = [g] = 2\mu_\epsilon(\mathbb{R})$$
$$g(\infty) + g(-\infty) = 2d_\epsilon.$$
Thus $d_\epsilon = d$ is independent of $\epsilon$ as is $\mu_\epsilon(\mathbb{R}) = [g]/2$.  We now let $\epsilon = 1/n \rightarrow 0$.  We consider the measure $\mu_\epsilon$ as a measure on the two point compactification of $\mathbb{R}$ obtained by adding $\pm \infty$.  The function $\tanh$ extends to a continuous function on this compactification and we extend the measure $\mu_\epsilon$ to $[-\infty,\infty] $ by assigning zero measure to $\pm \infty$.  There is a subsequence of these measures that converge weakly to a measure $\mu$ which assigns say $c_\pm $ to $\pm \infty$.  Thus we have 

$$g(x) = \int_{\mathbb{R}}  \tanh \hat{a}(x-t) d\mu(t) + d+ c_- - c_+.$$  Thus $g\in K_a$.  If we have instead  $\int e^{-b \xi} f'(\xi) < \infty $ for all $b$ with $0\le b <2a$ define $\tilde{f}(\xi) = - f(-\xi)$.  Then by complex conjugation, $i[\tilde{f}(P),g(Q)] \ge 0$.  The result then follows from what we have just shown with $f$.  This completes the proof.
\end{proof}

\section{Finite rank K}

In this section we assume that $K$ is finite rank and thus 

$$K = \sum_{k=1}^N (\phi_k,\cdot)\phi_k$$
where the $\phi_k$ are linearly independent functions in $L^2(\mathbb{R})$.  Similarly $\tilde{K} : = \mathcal{F} K \mathcal{F}^{-1}$ has the representation
\begin{equation} \label{Ktilde}
    \tilde{K} = \sum_{k=1}^N (\widehat{\phi_k},\cdot)\widehat{\phi_k}.
\end{equation}
In \cite{HK} it was shown that that $\phi_k, f',g' \in \mathcal{S}$, the Schwartz space.  In this section we will show that in fact they are analytic in strips centered on the real axis. 

\begin{theorem}\label{KaKb}
Assume without loss that $f$ and $g$ are increasing.  Under the finite rank assumption there exist positive $r'$ and $r$ with $f\in K_{r'}$ and $g\in K_r$.  We have $\int g'(t) e^{2s|t|} dt < \infty$ for all $s<r'$ and $\int f'(t) e^{2s|t|}dt < \infty$ for all $s< r$.  The $\phi_k$ are analytic in the strip $S_r = \{|\mathrm{Im} z| < r\}$.  We have $rr'\le \pi/2$ and $\int |\phi_k(x+iy)|^2 e^{2s|x|} dx < \infty$ for $|y| < r, s < r'$.
\end{theorem}

\begin{proof}
    
Define the vector $v(x) = <\phi_1(x),...,\phi_N(x)>$ in $\mathbb{C}^N$.  Note that the set $\{v(y): y\in \mathbb R\}$ spans $\mathbb {C}^N$. Otherwise there is a non-zero vector $w \in \mathbb{C}^N$ orthogonal to all $v(x), x\in \mathbb{R}$ in which case $\sum \overline{w_j}\phi_j(x) = 0 $ for all $x \in \mathbb{R}$  which contradicts linear independence. 
Choose $y_1,y_2,...,y_N$ so that $\{v(y_1),...,v(y_N)\}$ is a set of linearly independent vectors in $\mathbb{C}^N$.  And define (with $(\cdot,\cdot)$ the inner product in $\mathbb{C}^N$),

$$\gamma_j(x) = \frac{1}{\sqrt{2\pi}}\frac{g(x)-g(y_j)}{x-y_j}\widehat{f'}(y_j-x) = (v(y_j),v(x)) = \sum_k\phi_k(x)\overline {\phi_k}(y_j)$$

The matrix $M_{kj} = \overline{\phi_k(y_j)}$ is invertible since $\sum_j \overline{M_{kj}}c_j = 0$ for a non-zero vector $c$ implies $\sum _j c_jv(y_j) = 0$ which is not possible since the $v(y_j)$ are linearly independent. Thus is is easy to go back and forth between $\gamma_j$ and $\phi_j = \sum_k\gamma_k M^{-1}_{kj}$.  We will need a disjoint set of $N \ y_j$'s, call them $y_j'$ such that the $v(y_j')$ are also linearly independent.  Such a set can be found since $v(y)$ is continuous.  Thus we have both $\gamma_j$ and $\tilde{\gamma_j}$ with the latter constructed using the $y_j'$. The two sets of $y_j$'s helps to deal with the apparent singularity when $x=y_j$.   
We proceed by estimating derivatives using the fact that $f'$, $g'$, and the $\phi_j$'s are in $\mathcal{S}(\mathbb{R})$. Consider 
$$\gamma_j^{(n)}(x)  = \frac{1}{\sqrt{2\pi}}\sum_{k=0}^n \binom{n}{k}\Big(\frac{g(x) - g(y_j)}{x-y_j}\Big)^{(k)} i^{(n-k)} \widehat {\xi^{n-k}f'}(y_j-x)$$

We have $(g(x) - g(y))/(x-y) = \int_0^1 g'(x + t(y-x)) dt$.  Thus 

$$(d^k/dx^k)\frac{g(x) - g(y)}{x-y} = \int_0^1 (1-t)^k g^{(k+1)}(x +t(y-x))dt.$$  
We integrate by parts to find 

$$\int_0^1 (1-t)^k g^{(k+1)}(x + t(y-x))dt = (y-x)^{-1}\int _0^1 (g^{(k)}(x+t(y-x))-g^{(k)}(x)) k(1-t)^{k-1}dt$$ unless $k=0$ in which case we are back to $(g(y) - g(x))/(y-x)$. (Note the above even makes some sense when $k = 0$ (at least in the limit) since as $k\to 0 , k(1-t)^{k-1}$ approaches a delta function at $1$.)   
Thus we have

$$|\big((g(y) - g(x))/(x-y)\big)^{(n)}| \le 2|y-x|^{-1} ||g^{(n)}||_\infty.$$
$$||\big((g(y) - g(x))/(x-y)\big)^{(n)}1_{|x-y| >b}||_2 \le (2\sqrt{2}) b^{-1/2} ||g^{(n)}||_\infty.$$
We have 

$$|\gamma_j^{(n)}(x)| \le \frac{1}{\sqrt{2\pi}}2|x-y_j|^{-1} \sum_{k\le n}\binom{n}{k} ||g^{(k)}||_{\infty}||\widehat{\xi^{n-k}f'}||_\infty $$

In computing $||\phi_k^{(n)}||_\infty$ we can use $\gamma_j$ or $\tilde \gamma_j$.  Thus we must bound the minimum of $\sum_j |x-y_j|^{-1}$ and $\sum_j |x-y'_j|^{-1}$.  If  $S = \{y_1,...y_N \}$ and  $S' = \{y'_1,...y'_N \}$, we have $|x-y_j| \ge d(S, x)$ thus $\sum_j |x-y_j|^{-1} \le N/d(S, x)$.  Hence the minimum of the two sums is $\le N \min \{d(S, x)^{-1}, d(S', x)^{-1}\} \le 2Nd(S,S')^{-1}$. Thus we have 

\begin{equation} \label{phibound}
||\phi_j^{(n)}||_\infty \le C \sum_{k\le n}\binom{n}{k} ||g^{(k)}||_{\infty}||\widehat{\xi^{n-k}f'}||_\infty
\end{equation}
\begin{align*}
    &||\gamma_j^{(n)}1_{d(S, x)>b}||_2 \le ||\gamma_j^{(n)}1_{|x-y_j| >b}||_2 \le \\ 
&(2\sqrt{2}) b^{-1/2}\sum_{k\le n}\binom{n}{k} ||g^{(k)}||_\infty ||\widehat{\xi^{n-k}f'}||_\infty
\end{align*}
Since the intersection $\{d(S, x) \le b\}\cap d(S', x) \le b\}$ is null for small enough $b$, we have 
\begin{equation}
||\phi_j^{(n)}||_2  \le C\sum_{k\le n}\binom{n}{k} ||g^{(k)}||_\infty ||\widehat{\xi^{n-k}f'}||_\infty
\end{equation}\label{phiL2bound}

Using the equation 

\begin{equation} \label{gprime}
g'(x) = (2\pi/[f])\sum_j |\phi_j(x)|^2
\end{equation}
we obtain
\begin{equation}\label{gbounds}
||g^{(l+1)}||_\infty \le (2\pi/[f]\sum_j\sum _k \binom {l}{k} ||\phi_j^{(l-k)}||_\infty ||\phi_j^{(k)}||_\infty.
\end{equation}
$$||g^{(l+1)}||_1\le (2\pi/[f]\sum_j\sum _k \binom {l}{k} ||\phi_j^{(l-k)}||_2||\phi_j^{(k)}||_2.$$
From (\ref{tildekernel}) and (\ref{Ktilde})  we obtain 
$$f'(\xi) = (2\pi/[g])\sum_j|\widehat{\phi_j}(\xi)|^2.$$
Taking the Fourier transform we get 
$$\widehat{f'}(y-x) = (\sqrt{2\pi}/[g])\sum_j\int\phi_j(u+x)\overline{\phi_j(u+y)}du$$
After differentiating $k$ times and integrating by parts $l\le k$ times we learn that 
$$d^k/dx^k \widehat {f'}(y-x) = (-1)^l(\sqrt{2\pi}/[g]) \sum_j \int\phi_j^{(k-l)}(x+ u) \overline {\phi_j^{(l)}}(y+u)du$$ 
so that 

\begin{equation} \label{f'bound}
 ||\widehat {\xi^kf'}||_\infty \le (\sqrt{2\pi}/[g]) \sum_j ||\phi_j^{(k-l)}||_2 ||\phi_j^{(l)}||_2.
 \end{equation}

With $K\ge 1$ and $a\ge 1$ inductively assume 

\begin{equation}\label{induction}
||\phi_j^{(l)}||_\infty \le a K^l l! \ \text{for} \ l\le n-1 \ \text{and the same estimate for} \ ||\phi_j^{(l)}||_2. 
\end{equation}

\flushleft Using the estimate (\ref{f'bound}) for $\widehat {\xi^{n-k}f'}$ and the inductive assumption (\ref{induction}) we have for $l\le n-k$, 

\begin{equation}\label{f'bound2}
||\widehat {\xi^{n-k}f'}||_\infty \le (\sqrt{2\pi}/[g])N a^2 K^{n-k}l!(n-k-l)!
\end{equation}
as long as $l \le n-1$ and $n-k-l \le n-1$.  We use the $n!$ bounds 

$$\sqrt{2\pi n}(n/e)^n < n! < \sqrt{2\pi n}(n/e)^ne^{1/12n}$$.\
\flushleft in (\ref{f'bound2}). For $n-k$ even we take $l = (n-k)/2$. Otherwise take $l = (n-k -1)/2$. This is allowed if $n\ge 3$.  We obtain

\begin{equation} \label{f'bound3}
||\widehat {\xi^{n-k}f'}||_\infty \le K^{n-k}(\sqrt{2\pi}/[g])N a^2 C_0 \sqrt{(n-k +1)}(n-k)! /2^{n-k} 
\end{equation} \

where $C_0$ is independent of $n$ and $k$.  We choose $C_0$ large enough so that this estimate holds for all $n,k \ge 0$.

From (\ref{gbounds}) and (\ref{induction}) we have for $n \ge k\ge 1$

\begin{equation}\label{gbounds2}
||g^{(k)}||_\infty \le (2\pi/[f])Na^2K^{k-1}(k-1)!
\end{equation} \

We substitute the bounds (\ref{gbounds}) and (\ref{f'bound3}) into (\ref{phibound}) using the inductive assumption (\ref{induction}). We separate out the $k=0$ term below.  Let $C_1 = CC_0N(\sqrt{2\pi}/[g])$, and $C_2 = (2\pi/[f])N C_1$. We have \
\begin{align*}
&||\phi_j^{(n)}||_\infty \le C_1a^2n!\sum_{k\le n}||g^{(k)}||_\infty\sqrt{n-k+1}(K/2)^{n-k}/k!\\
& \le C_2a^4K^{n-1}n! \sum_{1\le k\le n} \sqrt{n-k+1}/k2^{n-k} + C_1a^2||g||_\infty K^n n!\sqrt{n+1}/2^n \\
&\le C_3a^4K^{n-1}n! + C_1a^2||g||_\infty K^n n!\sqrt{n+1}/2^n
\end{align*}
where $C_3 = C_2\sum_{m=0}^\infty \sqrt{m+1}/2^m$.
Choose $a$ so that $||\phi_j||_\infty \le a$, $||\phi_j||_2 \le a$.  (Note (\ref{phibound}) for $||\phi_j||_\infty$.) Choose $n_0$ so that $C_1a||g||_\infty2^{-n}\sqrt{n+1} \le 1/2$ for $n \ge n_0$.  Choose $K_0 >1 $ so that $||\phi_j^{(n)}||_\infty \le aK_0^n n!$, and $||\phi_j^{(n)}||_2 \le aK_0^nn! $ for $n\le n_0$.  Then for any $K\ge K_0$ and $K^{-1}C_3a^3 \le 1/2$, $||\phi_j^{(n)}||_\infty \le a K^n n!$ for all $n$.  Similarly for $||\phi_j^{(n)}||_2$. This completes the induction and thus $\phi_j$ has an analytic continuation to the strip $S_r$ with $r = K^{-1}$, bounded and in $L^2$ of any smaller strip.

\begin{lemma} \label{exponential decay}
Let $r$ be as above.  Then if $s<2r$, $\int e^{s|\xi|}f'(\xi) d\xi < \infty$.
\end{lemma}
\begin{proof} According to (\ref{f'bound3}) we have $|\widehat {\xi^n f'} (0)| \le C_4(K/2)^n \sqrt {n+1}n!$.  Thus  $$\int \sum_{n=0}^M ((s\xi)^{2n}/(2n)!) f'(\xi) d\xi \le C_4\sum_{n=0}^\infty (s/2r)^{2n}\sqrt{2n + 1}< \infty.$$  Thus by the monotone convergence theorem (since $f' \ge 0$) we have $\int \cosh(s\xi)f'(\xi)d\xi < \infty$.
\end{proof}

From Theorem \ref{expdecay} it follows that $g\in K_r$.  \\
\vspace{.3cm}
Let us complete the proof of Theorem \ref{KaKb}.  We notice that because of the symmetry between $f$ and $g$ we have the existence of an $r'>0$ so that $\hat\phi_j$ and $f$ have  analytic continuations to $S_{r'}$ with $f\in K_{r'}$ and $\int g'(x)e^{2s|x|} dx <  \infty$ when $s<r'$. 

From Kato's representation theorem there is a finite positive measure $\mu_r$ and real constant $c_r$ such that 

$$g(x) = \int\tanh\hat r (x-t) d\mu_r(t) + c_r $$ where $\hat r = \pi/2r$.

Thus $g'(x) = \hat r\int (\cosh(\hat r(x-t))^{-2} d\mu_r(t)$ and $$\int g'(x)e^{2sx} dx = \hat r\int(\cosh(\hat rx))^{-2}e^{2sx}dx \int e^{2st} d\mu_r(t) < \infty $$ for $|s| < r'$.  This implies $r'\le \hat r$ and $\int e^{2s|t|} d\mu_r(t)$
for $s<r'$.  In particular, since $\hat r = \pi/2r$, $rr'\le \pi/2$, which is the last statement of Theorem \ref{KaKb}.

Of course since $g\in K_{r_1}$ if $r_1 < r$  we can also write $g(x) = \int \tanh \hat r_1(x-t) d\mu_{r_1}(t) +c_{r_1} $.  We have

$$w*\lim_{y\uparrow r} \text{Im}g(x+iy)dx = 2rd\mu_r(x)$$.  

(This formula differs from the one given in \cite{TK} by an unimportant factor ). We learn that $\int e^{2s|x|} d\mu_{r_1}< \infty$ if $s< r'$.  If we take $0<r_1 < r$ we have as before $\text{Im}g(x+ir_1) \hat f'(-2ir_1) = r_1 \sum_j|\phi_j(x+ir_1)|^2$

Thus $$\infty > \int e^{2s|x|}d\mu_{r_1}(x) = r_1\sum_j(2r_1)^{-1}\int e^{2s|x|}|\phi_j(x+ir_1)|^2 dx(\hat{f'}(-2ir_1))^{-1}$$ 

for $s< r'$.  By varying $r_1$ we learn that 

$$\int|\phi_j(x+iy)|^2 e^{2s|x|}dx < \infty ; |y| <r, s < r'.$$

This completes the proof of Theorem \ref{KaKb}.
\end{proof}

In Theorem \ref{KaKb} we do not know how large $r$ and $r'$ are.  But what we do know (see the proposition below) is that if $f$ and $g$ are analytic and bounded in larger strips than we have estimated, then their imaginary parts are positive in the corresponding upper half strips:
\begin{proposition}\label{larger rr'}
Suppose $f$ and $g$ are as in Theorem \ref{KaKb}.  If $t \ge r$ and $t' \ge r'$ and $g$ and $f$ are analytic in $S_{ t}$ and $S_{ t'}$ respectively and bounded in any smaller strips, then $g\in K_{t}$ and $f\in K_{t'}$.  In addition $\int g'(x)e^{2s|x|} dx < \infty$ for $s < t'$ and $\int f'(\xi)e^{2s|\xi|}d\xi < \infty$ if $s < t$. We have $tt' \le \pi/2$.
\end{proposition}
 
\begin{proof} If we can show $\int f'(\xi)e^{s|\xi|} d\xi < \infty $
if $s< 2t$, $g\in K_t$ would follow from Proposition \ref{expdecay}. We have   $\int f'(\xi)e^{s|\xi|}d\xi < \infty $ if $s < 2r$ and  
\begin{equation}\label{gammaj}
    \gamma_j(x) = \frac{g(x) - g(y_j)}{x-y_j}\widehat{f'}(y_j-x).
\end{equation} 
We can analytically continue the right hand side of (\ref {gammaj}) into $S_{r_1}$ where $r_1 = \min\{t,2r\}$.  Thus $\phi_j$ is analytic in $S_{r_1}$ and is bounded in $S_{r_1 -\epsilon}$ with $\int |\phi_j(x+iy)|^2 dx \le C_\epsilon$ for all $|y| \le r_1 - \epsilon$.  It also follows that $|\phi_j(x+iy)| \rightarrow 0$ as $|x| \to \infty$ uniformly for $|y| < r_1 -\epsilon$.  Let $\hat{h}\in C_0^{\infty}(\mathbb{R})$.  By Cauchy's  theorem we have 
$$0 = \int_{-L}^L\bar{h}(x)\phi_j(x+iy)dx - \int_{-L}^L \bar{h}(x+iy)\phi_j(x)dx   $$
$$+  i\int_0^y[\bar{h}(-L+iu)\phi_j(-L -i(u-y)) -\bar{h}(L+iu)\phi_j(L -i(u-y)) ]du.$$
The integrals from $0$ to $y$ vanish in the limit $L\to \infty$ so that 
$$\int\bar{h}(x)\phi_j(x+iy)dx = \int \bar{h}(x+iy)\phi_j(x)dx   $$ or writing $\phi_{j,y}(x) = \phi_j(x+iy)$ and using the Plancherel theorem
$$\int\bar{\hat{h}}(\xi)\widehat{\phi_{j,y}}(\xi)d\xi = \int\bar{\hat{h}}(\xi)e^{-y\xi}\hat{\phi_j}(\xi)d\xi.$$  It follows that 
$$\widehat{\phi_{j,y}}(\xi)=e^{-y\xi}\hat{\phi_j}(\xi).$$
Since this holds for $|y|< r_1$ we see that in particular $ e^{s|\xi|}\hat{\phi_j} \in L^2$ for $s < r_1$. Since 
$$f'(\xi) = (2\pi/[g])\sum_j|\widehat{\phi_j}(\xi)|^2$$ it follows that $\int f'(\xi)e^{s|\xi|} d\xi < \infty $
if $s< 2r_1$ which is an improvement over $2r$ as long as $t > r$.  The proof is now complete if $t\le 2r$.  Otherwise we use (\ref{gammaj}) again to show that $\phi_j$ has an analytic continuation to $S_{r_2}, r_2 = \min\{t, 2r_1\}$ which is bounded in  $S_{r_2 -\epsilon}$ with $\int |\phi_j(x+iy)|^2 dx \le C_\epsilon$ for all $|y| \le r_2 - \epsilon$.  Thus continuing we find that if $n$ is the smallest integer such that $2^n r \ge t$, and $m \le n$, setting $r_m = \min\{t, 2^m r\}$ (so that $r_n = t$) , we obtain $\int |\phi_j(x+iy)|^2 dx \le C'_\epsilon$ for all $|y| \le r_m - \epsilon$  and finally $\int f'(\xi)e^{s|\xi|}d\xi < \infty$ if $s < 2t$. The result now follows from Proposition \ref{expdecay}.
\end{proof}

We remark that our method of proof Theorem \ref {KaKb} cannot predict the correct size of the strips.  This is because the positivity of the commutator is not used in our proof of analyticity but rather only to prove the positivity of the imaginary part of $g$ (or $f$) in the upper half strip of analyticity.  If we just assume the finite rank property, then the proof of analyticity still works as given here but the following is an example where the commutator has finite rank but is neither positive nor negative and has $f \in K_a$ and $g \in K_b $  with $0 < ab< \pi/2$.  
Take $g(x) = \tanh x$ and $f(\xi) = \tanh (\alpha_1 \xi) + \beta \tanh (\alpha_2 \xi)$  where $\alpha_1 = \pi/2$, $\alpha_2 = \pi$, and $\beta$ is positive.  The rank three commutator $i[f(P),g(Q)]$ can be written 
$$ i[f(P),g(Q)] = \pi^{-1}(\phi, \cdot)\phi + (\beta/\pi)[(\phi_+, \cdot) \phi_+ - (\phi_{-},\cdot)\phi_{-} ]$$
where $$\phi(x) = (\cosh x)^{-1}, \phi_{+} = \frac{\cosh(x/2)}{\cosh x}, \phi_{-} = \frac{\sinh(x/2)}{\cosh x}.$$
One sees $(\phi, \phi_{-}) = 0$ and thus $(\phi, i[f(P),g(Q)]\phi) > 0$ but 
$$ (\phi_{-}, i[f(P),g(Q)]\phi_-) = -(\beta/\pi)||\phi_-||^4 < 0.$$
The function $g$ is analytic in $S_{\pi/2}$ while $f$ is analytic in $S_{1/2}$.  The product is $(\pi/2)(1/2) = \pi/4$, half the number required for positivity if the Kato conjecture is true.  We also have $\text{Im}g(z) > 0$ for $0< \text{Im} z < \pi/2$ and $\text{Im}f(z)> 0$ for $0< \text {Im} z < 1/2$. But of course the commutator is neither positive nor negative.  We have 
$$ \tanh x = \int \tanh\hat r (x-t) d\mu(t)$$ and 
$$\tanh \alpha_1 \xi + \tanh \alpha_2 \xi = \int \tanh \hat r'(\xi - \eta) d\nu(\eta)$$
where $\mu$ and $\nu$ are finite positive measures (in fact $\mu$ is just a point mass of $1$ at $0$) and $\hat r = 1, \hat r' = \pi$ so that $rr' = \pi/4$.
\vspace{.2cm}

\section{We can assume f is entire and in some $K_a$}  
In this section we prove Theorem \ref{$K_a$ assumption} restated for convenience below:  
\begin{theorem} 
Suppose that the Kato conjecture is true for all entire $f$ and $g$ with $K\ge 0$ and with the additional constraint that $f\in K_{a'}$ and $g\in K_{b'}$ for some positive $a'$ and $b'$.  Then the Kato conjecture is true in general.
 \end{theorem}
\begin{proof}
Suppose $K =i[f(P),g(Q)] \ge 0, K\ne 0$ where $f$  and $g$  are real, bounded, and measurable. We know now that $f$ and $g$ are absolutely continuous with derivatives which we can assume positive a.e.  
Let $\phi_\epsilon(t) = (2\epsilon)^{-1} \cosh^{-2}(t/\epsilon)$.  Then $\int \phi_\epsilon(t) dt = 1$.  Define $$f_\epsilon = \phi_\epsilon * f.$$ 
We claim $f_\epsilon \in K_{\pi\epsilon/2}.$ To see this note 
$$\phi_\epsilon * f(t) = -\int f(s) (d/ds)(1/2)\tanh\epsilon^{-1}(t-s)ds .$$
Integrating by parts we have
$$f_\epsilon(t) = \int \tanh\epsilon^{-1}(t-s) d\mu(s) + c$$
where $d\mu(s) =(1/2) f'(s) ds$ and $c= (f(\infty) +f(-\infty))/2.$  Thus $f_\epsilon \in K_{\pi \epsilon/2}$. Similarly we define $g_\delta = \phi_\delta * g$ so that $g_\delta \in K_{\pi \delta/2}.$ 
Let $\psi(t) = (2\pi)^{-1/2} e^{-t^2/2}$ and $\psi_\alpha(t) = \alpha^{-1} \psi(t/\alpha)$.  Define $f_{\epsilon,\alpha} = \psi_{\alpha}*f_\epsilon$.  Then as can be seen by the commutativity of the $*$ operation, $f_{\epsilon,\alpha} \in K_{\pi\epsilon/2}$ and entire. Similarly $g_{\delta,\beta} = \psi_{\beta}* g_\delta \in K_{\pi\delta/2 }  $ and entire. 
By assumption the Kato conjecture is true for $f_{\epsilon,\alpha}$ and $g_{\delta,\beta} $.  (Note that of course $i[f_{\epsilon,\alpha}(P),g_{\delta,\beta}(Q)] \ge 0$.)
Then we know that $f_{\epsilon,\alpha} \in K_a$ and $g_{\delta,\beta} \in K_b$ for some $a>0$ and $b>0$ with $ab=\pi/2$.  Let us fix $\epsilon, \delta, \beta$ and  and let $\alpha \rightarrow 0 $. We have $ab=a(\alpha) b = \pi/2$ so that in fact $a$ is fixed.
We have 
$$f_{\epsilon,\alpha}(t) = \int \tanh\hat a(t-s) d\mu_{\epsilon,\alpha}(s) + c_{\epsilon,\alpha}.$$
Note that for each $\epsilon$ and $\alpha$, $\lim_{t\to \pm\infty} f_{\epsilon,\alpha}(t) = f(\pm\infty)$ and thus $f(\pm\infty) = \pm \mu_{\epsilon,\alpha}(\mathbb{R}) + c_{\epsilon,\alpha}.$  In particular $c_{\epsilon,\alpha}$ and $\mu_{\epsilon,\alpha}(\mathbb{R})$ do not depend on $\alpha$. We compactify $\mathbb{R}$ by adding the points $\pm \infty$ and we set $\mu_{\epsilon,\alpha} (\{\pm \infty\}) = 0.$  Then there is a measure $\mu_\epsilon$, a weak $*$  limit of the $\mu_{\epsilon,\alpha}$ as $\alpha \rightarrow 0$, and a constant $c = c_{\epsilon,\alpha}$ so that
$$f_\epsilon(t) = \int_{\mathbb{R}}\tanh \hat a(t-s) d\mu_\epsilon(s) +  d_\epsilon $$ where $d_\epsilon = \mu_\epsilon(\{-\infty\}) - \mu_\epsilon(\{\infty\})+ c.$

Thus $f_\epsilon\in K_a$.  Similarly $g_\delta \in K_b$.  The same analysis allows us to take the limit $\epsilon \rightarrow 0$ and the limit $\delta \rightarrow 0$ so that $f\in K_a$ and $g\in K_b$ and Kato's conjecture is proved.
\end{proof}


\begin{thebibliography}{99}

\bibitem{TK} T. Kato, \emph {Positive commutators $i[f(P),g(Q)]$}, J. Functional Anal. $\mathbf{96}$, (1991), 117--129.

\bibitem{JSH} J. Howland, \emph{Perturbation theory of dense point spectrum}, J. Functional Anal. $\mathbf{94}$, (1987), 52--80.

\bibitem{HK}  I. Herbst and T. Kriete, \emph{The Howland - Kato Commutator Problem}, in Analysis and Operator Theory; Dedicated to the Memory of  Tosio Kato's 100th Birthday, Rassias, M. and Zagrebnov, V. (Eds.), Springer, (2019), 191--223.

\bibitem{W}  D. V. Widder, \emph{Functions harmonic in a strip}, Proc. AMS, $\mathbf{12}$, (1961), 67--72.

\bibitem{B1} C. Brislawn, \emph{Kernels of trace class operators}, Proc. AMS, $\mathbf{104}$, No. 4, (1988), 1181--1189.

\bibitem{B2} C. Brislawn, \emph{Traceable integral kernels on countably generated measure spaces}, Pac. J. Math., $\mathbf{150}$, No. 2, (1991), 229--240.  

\bibitem{D} M. Duflo, \emph{G\'en\'eralit\'es sur les repr\'esentations induites} in Repr\'esentations des Groupes de Lie R\'esolubles, Monographies de la Soc. Math. de France,
$\mathbf{4}$, Dunod, Paris, (1972), 93--119. 

\bibitem{L}  K. L\"owner, \emph{\"Uber monotone Matrixfunktionen} (German), Math. Z. $\mathbf{38}$, No. 1, (1934), 177--216.
















\end{thebibliography}
\end{document}